\documentclass[reqno,english]{amsart}

\usepackage{amsmath,amsfonts,amssymb,graphicx,amsthm,enumerate,url}
\usepackage[noadjust]{cite}
\usepackage{stmaryrd}
\usepackage{comment,paralist,mathrsfs}
\usepackage{mathrsfs,booktabs,tabularx}
\usepackage{xifthen,xcolor,tikz,setspace}
\usetikzlibrary{decorations.pathmorphing,patterns,shapes,calc,decorations}
\usetikzlibrary{decorations.pathreplacing}
\usepackage{mathtools}
\usepackage[small]{caption}
\usepackage{subcaption}

\usepackage[colorinlistoftodos]{todonotes}
\usepackage[colorlinks=true]{hyperref}

\numberwithin{equation}{section}

\renewcommand{\epsilon}{\varepsilon}

\usepackage{color}

\newtheorem{maintheorem}{Theorem}

\newtheorem{conjecture}{Conjecture}

\newtheorem{theorem}{Theorem}[section]
\newtheorem*{theorem*}{Theorem}
\newtheorem{lemma}[theorem]{Lemma}

\newtheorem{claim}[theorem]{Claim}

\newtheorem*{observation*}{Observation}
\newtheorem{fact}[theorem]{Fact}

\newtheorem{remark}[theorem]{Remark}

\theoremstyle{definition}{

\newtheorem{definition}[theorem]{Definition}
\newtheorem*{definition*}{Definition}

}

\newcommand{\E}{\mathbb E}

\renewcommand{\P}{\mathbb P}

\newcommand{\R}{\mathbb R}

\newcommand{\bv}{\mathbf r}

\newcommand{\p}{{\bf p}}
\newcommand{\q}{{\bf q}}

\DeclareMathOperator{\var}{Var}

\DeclareMathOperator{\rank}{rank}
\begin{document}

\title{On the Rigidity of Random Graphs in high-dimensional spaces}
\author{Yuval Peled}
\address{Einstein Institute of Mathematics\\
 Hebrew University\\ Jerusalem~91904\\ Israel.}
\email{yuval.peled@mail.huji.ac.il}
\thanks{Y. Peled was partially supported by the Israel Science Foundation grant ISF-3464/24.}
\author{Niv Peleg}
\address{Einstein Institute of Mathematics\\
 Hebrew University\\ Jerusalem~91904\\ Israel.}
\email{niv.peleg@mail.huji.ac.il}

\begin{abstract}
We study the maximum dimension $ {d}={d}(n,p)$ for which an Erd\H{o}s-R\'enyi  $G(n,p)$ random graph is $d$-rigid. Our main results reveal two different regimes of rigidity in $G(n,p)$ separated at $p_c=C_*\log n/n,~C_*=2/(1-\log 2)$ --- the point where the graph's minimum degree exceeds half its average degree. We show that if $ p < (1-\varepsilon)p_c $, then $ d(n,p) $ is asymptotically almost surely (a.a.s.) equal to the minimum degree of $ G(n,p) $. In contrast, if $ p_c \leq p = o(n^{-1/2}) $ then $ d(n,p) $ is a.a.s. equal to $ (1/2 + o(1))np $. The second result confirms, in this regime, a conjecture of Krivelevich, Lew, and Michaeli.

\end{abstract}
\maketitle

\section{Introduction}

A graph $G=(V,E)$ is called rigid in $\mathbb R^d$, or $d$-rigid, if for a generic embedding $\p:V\to\R^d$ (i.e., the $d|V|$ coordinates of $\p$ are algebraically independent over the rationals), every continuous motion of the vertices in $\R^d$ that starts at $\p$, and preserves the lengths of all the edges of $G$, does not change the distance between any two vertices. Asimow and Roth~\cite{AR78,AR79} discovered that this is equivalent to {\em infinitesimal} generic rigidity. In particular, a graph with more than $d$ vertices is $d$-rigid if the $|E| \times d|V|$ rigidity matrix of the framework $(G,p)$ is of rank $d|V|-\binom{d+1}2$ (see more in Subsection \ref{sub:rig}). It is easy to see that $1$-dimensional rigidity is equivalent to graph connectivity. In addition, Laman famously discovered a combinatorial characterization of planar rigidity ~\cite{Laman}. In contrast, for all $d>2$ the combinatorics of $d$-rigid graphs is far from being understood and the topic is studied extensively ~\cite{ book:Connelley-Simon,book:Graver-Servatius-Servatius,chapter:Jordan}.

The Erd\H{o}s-R\'enyi $G(n,p)$ random graph is one of the most well-studied models in random graph theory ~\cite{bollobas1998random,frieze2015introduction}. A graph in this model contains $n$ vertices, and each edge appears independently with probability $p=p(n)$. The theory studies the properties that occur in $G(n,p)$ asymptotically almost surely (a.a.s.), i.e., with probability tending to $1$ as $n\to\infty$. A monotone graph property has a sharp threshold at $p_c=p_c(n)$ if for every $\varepsilon>0$, $G(n,(1-\varepsilon)p_c)$ a.a.s. does not have the property, but $G(n,(1+\varepsilon)p_c)$ a.a.s. does have it.

Rigidity of random graphs had  been primarily studied in Euclidean spaces of a fixed dimension $d$. The goal of this work is to study the generic rigidity of a $G(n,p)$ random graph in $\mathbb R^d$ where the dimension $d$ diverges as $n\to\infty$. The question we investigate is: 

\begin{center}
{\em What is the largest dimension $d=d(n,p)$ for which $G(n,p)$ is $d$-rigid?}
\end{center}

Clearly, a graph cannot be $d$-rigid if $d$ is greater than its minimum degree. In addition, it is known that for a fixed dimension $d$, the sharp threshold probability for $d$-rigidity of $G\sim G(n,p)$ coincides with the threshold probability for the property that its minimum degree $\delta(G)$ is at least $d$. The case $d=1$ of this statement is precisely the seminal result of Erd\H{o}s and R\'enyi regarding the connectivity threshold of random graphs~\cite{ER}. The case $d=2$ was solved by Jackson, Servatius and Servatius~\cite{JSS-planeThreshold}. For larger fixed dimensions $d>2$, the problem was studied by Kir\'{a}ly and Theran~\cite{Kiraly-Theran:RigitidyThreshold}, Jord\'an and Tanigawa~\cite{Jordan-Tanigawa:RigitidyThreshold}, and was resolved by Lew, Nevo, the first author and Raz \cite{LNPR}. They showed that $G(n,p)$ is a.a.s. $d$-rigid if $np-(\log n+(d-1)\log\log n)\to\infty$ as $n\to\infty$, and that it is a.a.s. not $d$-rigid if this term tends to $-\infty$ .

This tells us that if $np=\log n + O(\log\log n)$,  then the bottleneck for rigidity of $G(n,p)$ lies in its minimum degree. On the other hand, Krivelevich, Lew and Michaeli ~\cite{krivelevich2023rigid} conjecture that if $np=\omega(\log n)$, then the bottleneck for rigidity of $G(n,p)$ lies in its {\em number of edges}. More precisely, it is clear that if $G\sim G(n,p)$ is a.a.s $d$-rigid then the inequality $p\binom n2\ge dn-\binom{d+1}2$ is satisfied, since the number of edges in $G$ concentrates around its mean. This reformulates to the inequality $d\le (1-o(1))n(1-\sqrt{1-p})$, and the conjecture is that this bound is sharp. 

\begin{conjecture}[Conjecture 1.11 in \cite{krivelevich2023rigid}]\label{conj:KLM}
    For every  $p=\omega(\log n /n)$, and every fixed $\varepsilon>0$, a random graph $G\sim G(n,p)$ is a.a.s $d$-rigid if $d<(1-\varepsilon)n(1-\sqrt{1-p})$.
\end{conjecture}

Among other things, their paper establishes, using strong $d$-rigid partitions, the weaker assertion that $G$ is a.a.s $(c\cdot np/\log{np})$-rigid for some constant $c>0$. 

\medskip

Given these two trivial bottlenecks for rigidity, the first step in our research was to compute the regimes of $p$ in which one bottleneck is more restrictive than the other. Namely, let $ p_c := C_* \log n/{n} $, where $ C_* := {2}/{(1 - \log 2)} \approx 6.518 $. It turns out that for every $\varepsilon > 0$, the inequality $\delta(G) < n(1 - \sqrt{1 - p})$ holds a.a.s.\ if $p < (1 - \varepsilon)p_c$, while the reverse inequality holds a.a.s. if $p > (1 + \varepsilon)p_c$. Further details can be found in Subsection \ref{sub:min} and Figure \ref{fig}.

In particular, for every fixed $\varepsilon>0$, if $p>(1+\varepsilon)p_c$ then $G$ a.a.s. does not have enough edges to be $\delta(G)$-rigid. Our first main theorem asserts that this is the only bottleneck for $\delta(G)$-rigidity in $G(n,p)$.

\begin{maintheorem}\label{thm:delta}
Fix $\varepsilon>0$, suppose $1-p=\omega(n^{-2})$, and let $G\sim G(n,p)$. Then,
\[
\P(G\mbox{ is $\delta(G)$-rigid}) \to
\left\{ 
\begin{matrix}
    1&p<(1-\varepsilon)\frac{C_*\log n}n \vspace{0.2cm} \\ 
    0&p>(1+\varepsilon)\frac{C_*\log n}n
\end{matrix}
\right.,
\]
as $n\to\infty$, where $C_*=2/(1-\log 2)$. 
\end{maintheorem}
Theorem \ref{thm:delta} proves, in particular, that (the complement of) $\delta(G)$-rigidity has a threshold in $G(n,p)$, even though it is not a monotone property. Indeed,
note that if $p=1-o(n^{-2})$ then $G=K_n$, and thus $\delta(G)-rigid$, with probability bounded away from $0$.

For $p\ge p_c$, we are able to prove Conjecture \ref{conj:KLM} in the regime where $p=o(n^{-1/2})$. Note that $1-\sqrt{1-p}=(1/2+o(1))p$ in this regime, since $p=o(1)$.

\begin{maintheorem}\label{thm:np/2}
Let $C_*=2/(1-\log 2)$, and suppose that $C_*\log n/n \le p= o(n^{-1/2})$.  If  $G\sim G(n,p)$ then, for every fixed $1/2>\varepsilon>0$ and integers $d=d(n)$,
\[
\P(G\mbox{ is $d$-rigid}) \to
\left\{ 
\begin{matrix}
    1&d<(1/2-\varepsilon)np\vspace{0.2cm} \\ 
    0&d>(1/2+\varepsilon)np
\end{matrix}
\right.,
\]
as $n\to\infty$.
\end{maintheorem}

\begin{figure}
    \centering
    \includegraphics[width=1\linewidth]{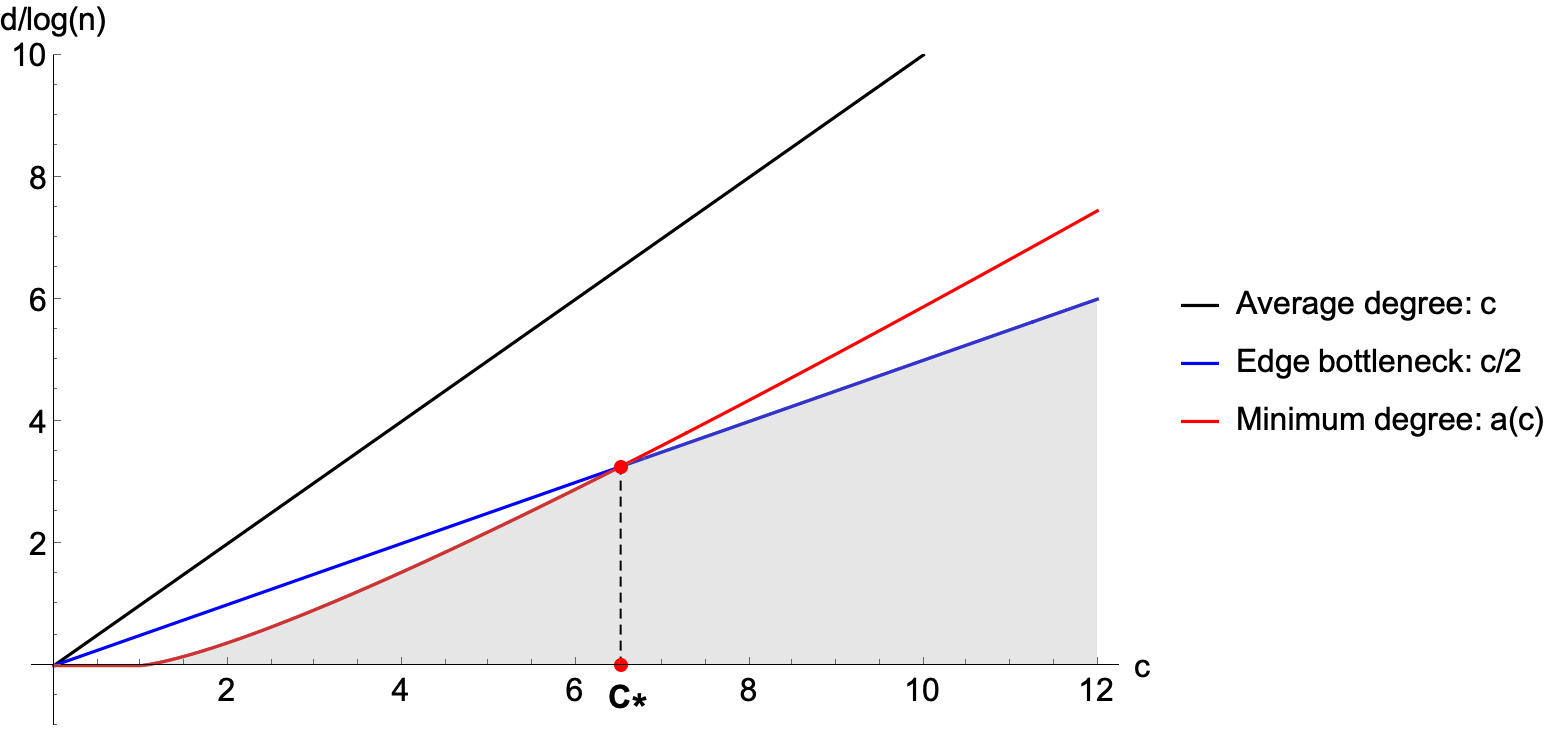}
    \caption{Rigidity of $G(n,p)$ where $p=c\log n/n$. Here the average degree is asymptotic to $c\log n$, the edge-number bottleneck is asymptotic to $(c/2)\log n$, and the $Y$-axis is normalized by $\log n$. The shaded area illustrates the values of $(c,d)$ for which $G(n,p)$ is a.a.s.\ $d$-rigid. If $c<C_*$ the minimum degree is the bottleneck for rigidity, and otherwise the edge number is the bottleneck.}
    \label{fig}
\end{figure}

The non-rigidity statements in both our theorems are immediate and follow directly from a dimension counting argument. The proofs of the rigidity statements are carried out using a similar strategy to ~\cite{LNPR}. First, we prove that the closure of $G$ in the rigidity matroid is dense. Subsequently, we deduce that the closure contains a large clique. The final step uses the combinatorial expansion properties of the random graph. Namely, we "bootstrap" the large clique in its closure to show it must be rigid. However, elements of this strategy that succeeded for fixed-dimensional rigidity break down when the dimension scales with the graph size, and two essential new ideas were needed.

First, for large $d$, we could not find an elementary way to derive the existence of a large clique in the closure from the assertion that it is dense. Fortunately, this can be derived using ideas from Vill\'any's beautiful proof of the Lov\'asz-Yemini conjecture ~\cite{villain}. Primarily, it is our use of  Vill\'any's argument that requires us assume that $p=o(n^{-1/2})$ in Theorem \ref{thm:np/2}. 

Second, as $\varepsilon$ approaches $0$, i.e., as $p$ approaches $C_* \log n /n$ in Theorem \ref{thm:delta} and as $d$ approaches $np/2$ in Theorem \ref{thm:np/2}, the lower bound our probabilistic method gives on the density of the closure deteriorates substantially. In consequence, the size of the clique in the closure we are able to guarantee is far too small to set off the final combinatorial ``bootstrap" part of our argument. To overcome this challenge, we extend our probabilistic approach and show that the closure is, roughly speaking, {\em evenly distributed} across the vertices. This strategy allows us to establish the existence of a very large clique in the closure. To achieve this goal, we work with many projections of the generic rigidity matrix.

\subsection{Generic global rigidity and reconstruction} A $d$-dimensional framework $(G,\p)$ is called {\em globally} $d$-rigid if every embedding $\p':V\to \R^d$ that realizes the pairwise distances $\|\p(v)-\p(u)\|$ for every edge $uv\in E$ is obtained from $\p$ by an isometry of $\R^d$. We say that a graph $G$ is globally $d$-rigid if $(G,\p)$ is globally $d$-rigid for some generic embedding $\p$. For a comprehensive introduction to global rigidity, its distinctions from rigidity and infinitesimal rigidity, and an overview of key developments in this area over recent decades, we refer the reader to~\cite{tanigawa2015sufficient, jordan2017global} and the references therein. For random graphs, it is shown in ~\cite{LNPR} that for a fixed $d$, a.a.s., a $G(n,p)$ graph is globally $d$-rigid if its minimum degree is $d+1$. In addition, non-generic global $1$-rigidity in random graphs has been studied recently~\cite{benjamini2022determining,girao2024reconstructing,montgomery2024global} in the context of reconstructing embeddings from random distance samples.

We recall two results from the theory of global rigidity that are needed here. First, the minimum degree of a globally $d$-rigid graph must be at least $d+1$ because a vertex with a smaller degree can be reflected across an affine hyperplane spanned by its neighbors. This produces a distinct embedding of the vertices that preserves the distances between adjacent pairs (see \cite{hendrickson1992conditions}). Second, Jord\'an proved that every $(d+1)$-rigid graph is globally $d$-rigid ~\cite{jordan2017Aglobal}. Therefore, the following theorems follow directly from our main theorems.

\begin{theorem}
Fix $\varepsilon>0$ and let $G\sim G(n,p)$. Then,
\[
\P(G\mbox{ is globally $(\delta(G)-1)$-rigid}) \to
\left\{ 
\begin{matrix}
    1&p<(1-\varepsilon)\frac{C_*\log n}n \vspace{0.2cm} \\ 
    0&p>(1+\varepsilon)\frac{C_*\log n}n
\end{matrix}
\right.,
\]
as $n\to\infty$.
\end{theorem}

\begin{theorem}
If $~C_*\log n/n \le p= o(n^{-1/2})$ and $G\sim G(n,p)$ then, for every fixed $1/2>\varepsilon>0$ and integers $d=d(n)$,
\[
\P(G\mbox{ is globally $d$-rigid}) \to
\left\{ 
\begin{matrix}
    1&d<(1/2-\varepsilon)np\vspace{0.2cm} \\ 
    0&d>(1/2+\varepsilon)np
\end{matrix}
\right.,
\]
as $n\to\infty$.
\end{theorem}

The remainder of the paper is organized as follows. In Section \ref{sec:prel} we present some necessary preliminaries in rigidity and random graphs. Afterwards, in Section \ref{sec:proofs} we prove the main theorems. We conclude with open problems in Section \ref{sec:open}.
\section{Preliminaries}\label{sec:prel}
\subsection{Generic infinitesimal rigidity}\label{sub:rig}

A $d$-dimensional framework is a pair $(G,\p)$ consisting of a graph $G=(V,E)$ and an embedding $\p:V\to\R^d$. Throughout this paper we assume that $|V|>d$ and that $\p$ is generic, i.e., the $d|V|$ coordinates of $\p$ are algebraically independent over the rationals. The $|E|\times d|V|$  rigidity matrix $R=R(G,\p)$ of the framework is the total derivative of the function
\[
\R^{d|V|}\to\R^{|E|}\,\,,\quad \q\mapsto \left(\frac 12\|\q(x)-\q(y)\|_2^2\right)_{xy\in E}
\]
at point $\q=\p$. Here we view a point $\q\in\R^{d|V|}$ as a function from $V\to\mathbb R^d$. More directly, the row $\bv_{xy}$ in $R$ indexed by $xy\in E$ contains the vector $(\p(x)-\p(y))^T$ in the $d$ coordinates corresponding to $x$, the vector $(\p(y)-\p(x))^T$ in the $d$ coordinates corresponding to $y$, and $0$ elsewhere. The tangent space at the identity of the isometry group of $\mathbb R^d$ yields vectors in the right kernel of $R$. Moreover, under the assumptions that  $|V|>d$ and $\p$ is generic, these vectors span a subspace of the kernel of dimension $\binom {d+1}2$ whence $\dim\ker R\ge\binom {d+1}2$. A graph $G$ is called $d$-rigid if this bound is attained, which in this case means that $\rank R=d|V|-\binom{d+1}2$. Asimow and Roth showed that if $\p$ is generic, this notion is a well-defined graph property, namely it is independent of the choice of the generic embedding $\p$ ~\cite{AR78,AR79}. If $|V|\le d$ then the rank of the rigidity matrix is at most $\binom {|V|}2$.

The $n$-vertex generic $d$-rigidity matroid is the linear matroid whose ground set consists of the edges of the complete graph $K_n$ and a subset $E$ is independent if the rows corresponding to $E$ are linearly independent in the rigidity matrix of $K_n$ with some generic embedding $\p$. The notion of {\em closure} from matroid theory is central in our work so we recall its definition. Recall that $\bv_e$ denotes the row corresponding to an edge $e$ in the rigidity matrix. 
\begin{definition}\label{def:closure}
Let $d\ge 1$ an integer, $G$ be an $n$-vertex graph, and $\p:[n]\to\R^d$ generic. The $d$-rigidity closure of $G$ is the $n$-vertex graph whose edge set is
\[
C_d(G) = 
\left\{
f \in \binom{[n]}2~:~ \bv_{f}\in\mathrm{span}_\R(\bv_e~:~e\in G).
\right\}
\]
A graph is called closed in the $d$-rigidity matroid if $G=C_d(G)$.
\end{definition}
Clearly, $G\subset C_d(G)$, and a graph is $d$-rigid if and only if $C_d(G)$ is the complete graph. In addition, $C_d(G)\supseteq C_{d'}(G)$ if $d<d'$ since the $d$-rigidity matrix is a submatrix of the $d'$-rigidity matrix.
\subsubsection{The structure of closed graphs}
This subsection contains a number of structural claims on closed graphs that are needed in our argument. Most notably, a critical part of our argument builds upon the following key lemma from the recent proof of the Lov\'asz-Yemini conjecture by Vill\'anyi~\cite{villain}. While the lemma is not explicitly stated in the paper in the exact form we require, its underlying idea is fully present.

\begin{lemma}\label{lem:vil}
Let $d\ge 1$ be an integer, and $G$ a closed graph in the $d$-rigidity matroid with minimal degree at least $d(d+1).$ Then, there is a vertex in $G$ whose neighbors induce a clique.
\end{lemma}
\begin{proof}
Suppose that $G$ is a closed graph in the $d$-rigidity matroid, and let $v$ be a vertex. Note that any two distinct maximal cliques $H_1,H_2$ in the subgraph induced by $v$'s neighbors intersect in at most $d-2$ vertices. Indeed, otherwise $H_1\cup \{v\}$ and $H_2\cup \{v\}$ are two cliques in $G$ that intersect in at least $d$ neighbors, whence $H_1\cup H_2\cup\{v\}$ is $d$-rigid. Since $G$ is closed, $H_1\cup H_2\cup\{v\}$ induces a clique in $G$, in contradiction to $H_1,H_2$ being maximal distinct cliques. Therefore, by Lemmas 3.1 and 3.2 in ~\cite{villain} either there is a vertex in $G$ whose neighbors induce a clique, or the random subgraph $G_\pi$ constructed by Vill\'anyi has, with positive probability, the contradictory property of being both independent in the $d$-rigidity matroid and having at least $d|V(G)|$ edges.
\end{proof}

The following fact follows directly from the fact that $K_{d+2}$ minus an edge is $d$-rigid.

\begin{fact}\label{fct:heavy_vertex}
Let $d\ge 1$ be an integer, and $G$ a closed graph in the $d$-rigidity matroid. Suppose that $A$ induces a clique in $G$, and $v\notin A$ is adjacent to at least $d$ vertices from $A$. Then, $A\cup\{v\}$ also induces a clique in $G$.
\end{fact}

For the proof of the next claim, that was suggested to us by Orit Raz, we recall the notion of Henneberg steps. Given a graph $G$ and a dimension $d$, a Henneberg $0$-step is the addition of a new vertex connected to $d$ vertices of $G$, and a Henneberg $1$-step is the addition of a new vertex connected to $d+1$ vertices $x_1,....,x_{d+1}$ of $G$, where $x_1,x_2$ are adjacent in $G$, followed by the removal of the edge $x_1x_2.$ It is known that these steps preserve $d$-rigidity.
\begin{claim}\label{clm:matching}
Let $d\ge 1$ be an integer, and $G$ a closed graph in the $d$-rigidity matroid. Suppose that $A$,$B$ are disjoint vertex subsets that induce a clique in $G$, and that there exists a matching of $\binom {d+1}2$ edges between $A$ and $B$ in $G$. Then, $A\cup B$ induces a clique in $G$. 
\end{claim}
\begin{proof}
    It suffices to construct a $d$-rigid graph $H$ on  $d(d+1)$ vertices, in which the edges between the first $\binom {d+1}2$ and the last $\binom {d+1}2$ vertices form a perfect matching. Indeed, the induced subgraph of $G$ on $A\cup B$ contains a copy of $H$ and since $G$ is closed and $H$ is $d$-rigid the claim follows.

    To construct $H$, we start with the complete graph on the vertices $x_1,...,x_d,y_1$. Then, we perform $d-1$ $0$-steps, where in the $i$-th step, $2\le i \le d$, we connect the new vertex $y_i$ to $y_1,...,y_{i-1},x_i,...,x_d$. Note that besides the two cliques induced on $x_1,...,x_d$ and $y_1,...,y_d$, this $d$-rigid graph contains the matching $x_iy_i,~1\le i \le d$ as well as $\binom {d+1}2-d$ additional edges $e_{d+1},...,e_{\binom {d+1}2}$. 

    We would like to remove these additional edges and replace them by new vertices that form a matching. So, for each such additional edge $e_k$ between $x_iy_j,~i\ne j$, we (i) add a new vertex $x_k$, connect it to $y_j,x_1,...,x_d$ and remove the edge $x_iy_j$, and afterwards (ii) add a new vertex $y_k$, connect it to $x_k,y_1,...,y_d$ and remove the edge $x_ky_j$. Note that these two actions are $1$-steps hence they preserve $d$-rigidity. In addition, once all these steps are carried out, the only edges between $x_1,....,x_{\binom{d+1}2}$ and $y_1,....,y_{\binom{d+1}2}$ are those of the perfect matching $x_ky_k~,1\le k \le \binom {d+1}2$.
\end{proof}

\subsubsection{Projections of the rigidity matrix}
Let $A\subseteq [n]$ be a vertex subset and denote by $W_A$ the orthogonal complement in $\R^{dn}$ of $V_A:=\mathrm{span}_\R\left(\bv_e~:~e\in \binom A2\right)$. In addition, denote the corresponding orthogonal projection by $P_{W_A}:\R^{dn}\to W_A$. Note that since $\dim V_A\ge d|A|-\binom{d+1}2$ (equality holds if $|A|\ge d$) then,
\begin{equation}\label{eq:dim_A}
    \dim \mathrm{span}_\R \left(P_{W_A}\bv_e~:~e\in \binom {[n]}2\setminus \binom A2\right) \le d(n-|A|).
\end{equation}

We will consider the closure in the associated linear matroid, 
\[
C_{d,A}(G) :=
\left\{
f \in \binom{[n]}2\setminus \binom A2~:~ P_{W_A}\bv_{f}\in\mathrm{span}_\R(P_{W_A}\bv_e~:~e\in G),
\right\}.
\]
Geometrically speaking, $C_{d,A}(G)$ consists of the pairs whose distance is preserved in a rigid motion of $G$ in which $A$ is fixed, not including the trivial pairs $\binom A2$. The following claim is derived easily.
\begin{claim}\label{clm:CdA}
If $G$ is an $n$-vertex graph and $A\subset [n]$ induces a clique in $C_d(G)$ then $C_{d,A}(G)\subset C_d(G)$.
\end{claim}
\begin{proof} 
Let $f$ be an edge that belongs to $C_{d,A}(G)$. There exists a linear combination
$
P_{W_A}\bv_f = \sum_{e\in G} \alpha_e P_{W_A}\bv_e.
$
Therefore, using $P_{W_A}=I-P_{V_A}$, we find that
\[
\bv_f = \sum_{e\in G} \alpha_e \bv_e + P_{V_A}\left(\bv_f - 
\sum_{e\in G} \alpha_e \bv_e\right).
\]
The claim follows by observing that (i) $P_{V_A}\bv_{e'} \in V_A$ for every $e'\in \binom{[n]}2$, and (ii) $V_A\subseteq \mathrm{span}_\R\{\bv_{e}: e\in G\}$ since $A$ induces a clique in $C_d(G)$.
\end{proof}

We note that this is a standard construction in matroid theory. Namely, the linear matroid that corresponds to $P_{W_A}\bv_e~:~e\in\binom{[n]}2\setminus\binom A2$ is the matroid obtained from the generic $d$-rigidity matroid by {\em contracting} all the edges in $\binom A2$.

\subsection{Chernoff's inequalities}
We make substantial use of Chernoff's inequalities in this work. Mostly, we use the following version that is asymptotically sharp in the parameter regime we are interested in. 

\begin{fact}
If $Y\sim\mathrm{Bin}(n,p)$ and $0<\alpha<1$ then 
\begin{equation}\label{eq:cher_main}
    \P(Y\le \alpha\E Y) \le 
    \left(\frac{e^{-(1-\alpha)}}{\alpha^\alpha}\right)^{\E Y}.
\end{equation}

\end{fact}
We will also use the following simpler version of this inequality:
\begin{equation}\label{eq:cher_non_sharp}
    \P(Y\le \alpha \E Y) \le 
    e^{-\frac{(1-\alpha)^2}{2}\E Y}.
\end{equation}

For upper tail bounds, we will only need that for every $t>\E Y$,
\begin{equation}\label{eq:cher_non_sharp_upper}
    \P(Y\ge t) \le 
    \left(\frac{e\E Y}{t} \right)^t.
\end{equation}

\subsection{Minimum degree in random graphs}\label{sub:min}

The main purpose of this subsection is to study the minimum degree of a $G(n,p)$ random graph in the regime $p=c\log n/n$ where $c>0$ is a constant. This is a well-known topic in random graph theory~\cite[Section 3]{bollobas1998random} and we include it here for completeness.
\medskip

For a real number $c> 1$, define the function
\begin{equation}\label{eq:phi}
    \varphi_c(t) := 1-c+t-t\log(t/c).
\end{equation}
and denote by $a=a(c)$ the smallest non-negative root of $\varphi_c$. Note that $0<a<c$ since $\varphi_c(t)$ approaches $1-c<0$ as $t\to 0$ and $\varphi_c(c)=1$.

Roughly speaking, the number of vertices of degree $t\log n$ in $G(n,c\log n/n)$ is typically $n^{\varphi_c(t)+o(1)}.$ The following claim is an important special case.

\begin{claim}
\label{clm: min_deg_small_p}
Suppose that $p=p_n$ are real numbers such that the sequence $c_n:=(n-1)p/\log n$ is bounded and satisfies $c_n\ge 1$ for every $n$. If $G \sim G(n,p)$ then, for every $\varepsilon>0$,
\[
\P(|\delta(G) - a(c_n)\log n| \le \varepsilon\log n)\to 1
\]
as $n\to\infty$.
\end{claim}

\begin{proof}
The proof is based on straightforward applications of the first and second moment methods. 
Let $\varepsilon>0$. Let $0<t<c=c_n$ be a real number, and $X$ the number of vertices of degree at most $t\log n$ in $G$. The degree of a vertex in $G$ is a binomial random variable $Y\sim\mathrm{Bin}(n-1,p)$, and therefore by Chernoff's inequality \eqref{eq:cher_main},
\[
\P_{\mathrm{Bin}(n-1,p)}(Y \le t\log n) \le \left(\frac{e^{-(1-t/c)}}{(t/c)^{t/c}}\right)^{c\log n}=n^{\varphi_c(t)-1}.
\]
In conclusion, using linearity of expectation, we find that 
\[
\E[X]\le n^{\varphi_c(t)}.
\]
 The lower bound follows directly since $t<a(c)-\varepsilon$ implies that $\E[X]\to 0$ as $n\to\infty$, and consequently $\delta(G)>t\log n$ a.a.s.\, as claimed.

 On the other hand, if $k:=\lfloor t\log n\rfloor=(1+o(1))t\log n$ then by Stirling's approximation and the inequality $1-x\ge e^{-(x+x^2)}$ for sufficiently small $x$, we find that
\begin{align*}
\P(Y=k)=&~\binom{n-1}kp^k(1-p)^{n-1-k}    \\
\ge&~\left((1+o(1))\frac{e(n-1)p}{k}\right)^ke^{-(p+p^2)(n-1-k)}\\
=&~n^{\varphi_c(t)-1+o(1)},
\end{align*}
by a straightforward computation using $k=(1+o(1))t\log n$ and $(n-1)p=c\log n$. Therefore, the expectation of the number $Z$ of vertices of degree $k$ is at least
$$
\E[Z]\ge n^{\varphi_c(t)+o(1)}\to\infty,
$$
assuming $t>a(c)+\varepsilon$.

In addition, if $D_x,D_y$ are the degrees of two vertices in $G$ then a direct computation shows that
\[
\P(D_x=k,D_y=k)=(1+o(1))\P(Y=k)^2,
\]
and consequently $\var (Z)= o(\E[Z]^2)$. We obtain that a.a.s. $Z>0$ by Chebyshev's inequality, and therefore $\delta(G) \le k\le (a(c)+\varepsilon)\log n$ a.a.s., as claimed.

\end{proof}

Using this claim, we can easily identify the critical probability at which the minimum-degree bottleneck for the rigidity of $G(n,p)$ becomes less restrictive than the bottleneck induced by the number of edges. Specifically, the edge-number bottleneck asserts that the typical $d$-rigidity of $G(n, c \log n / n)$ requires $d < (1 + o(1)) c \log n / 2$. Let $C_*$ denote the value of $c$ such that $a = c/2$, i.e.,
\[
\varphi_{C_*}(C_*/2)=1-C_*+(C_*/2)-C_*/2\log (1/2)=0,
\]
Solving this yields $C_* = 2 / (1 - \log 2) \approx 6.518$. Consequently, we find that $a < c/2$ if and only if $c < C_*$ (see Figure \ref{fig}).

\subsubsection{Matchings in random graphs}
We also include here the following  claim that is later used to bootstrap disjoint cliques into a larger clique in the closure of $G(n,p)$
\begin{claim}\label{clm:Gnp_matching}
Fix $\delta >0$. If $p=\omega(n^{-1})$ and $G\sim G(n,p)$, then, a.a.s., between every two disjoint vertex sets of sizes at least $\delta n$ there is a matching of at least $\delta n/2$ edges.  
\end{claim}
\begin{proof}
    A.a.s. there is at least one edge between every two disjoint vertex sets of sizes at least $\delta n /2$. Indeed, the expected number of disjoint vertex sets that violate this condition is at most
    \[
    \left(2^{n}\right)^2\cdot (1-p)^{\delta^2n^2/4} \le 4^n\cdot e^{-(\delta^2/4)\cdot pn^2}\to 0, 
    \]
    since $pn^2 = \omega (n)$. Therefore, we can construct the matching  by repeating the following greedy procedure $\delta n/2$ times: select an arbitrary edge between the sets and remove its endpoints.
\end{proof}

\section{Proofs of the main theorems}\label{sec:proofs}
\subsection{A well-spread dense closure}

The first step in our proofs is to establish that for the values of $d$ and $p$ that we consider, the $d$-rigidity closure of $G(n,p)$ is dense. Since the same argument is applied in our proof for the closures $C_{d,A}$ of other linear matroids, we state it in a more general setting. Note that the proof is a direct extension of ~\cite[Lemma 3.1]{LNPR}, and we include it here for the sake of completeness.

For a set $S$ and a number $p\in(0,1)$, denote by $S_p$ the random subset of $S$ containing each element of $S$ independently with probability $p$.
\begin{lemma}\label{lem:process}
Let $p,\varepsilon\in(0,1)$ be real numbers and $S$ a set of $m$ vectors. If $\dim\mathrm{span}(S)\le D$ then
\[
\P(|S\cap\mathrm{span}(S_p)|<\varepsilon m) \le \P(X\le D),
\]
where $X\sim\mbox{Bin}(m,(1-\varepsilon)p)$.
\end{lemma}
\begin{proof}
We sample $S_p$ in the following nonstandard $m$-step method. Let $r_1,...,r_m$ be independent and uniform in $[0,1]$. Initialize $B_0:=\emptyset$. At every step $i\in [m]$, we start by arbitrarily ordering the vectors in $S\setminus B_{i-1}$ such that vectors in $S\setminus\mathrm{span}(B_{i-1})$ appear first. Afterwards,
\begin{enumerate}
    \item if $r_i>p$, set $B_i:=B_{i-1}$, and
    \item otherwise, set $B_i:=B_{i-1}\cup\{v_i\}$, where $v_i$ is the $\lceil(r_i/p)\cdot |S\setminus B_{i-1}|\rceil$-th vector in the current ordering of $S\setminus B_{i-1}$.
\end{enumerate}
It is clear that the subset $B_m$ has the same distribution as $S_p$. Indeed, its size is Binomial with parameters $m,p$, and in every step in which a new vector is added, it is chosen uniformly at random.

In addition, let $E$ denote the event that $|S\cap\mathrm{span}(B_m)|<\varepsilon m$, and $E_i$ the event that $r_i<(1-\varepsilon)p$. We claim that $E$ is contained in the event that at most $D$ of the $E_i$'s occur. Indeed, $E$ implies that 
$|S\cap\mathrm{span}(B_{i-1})|<\varepsilon m$ whence if both $E$ and $E_i$ occur then
$$\lceil(r_i/p)\cdot |S\setminus B_{i-1}|\rceil \le  |S\setminus\mathrm{span}(B_{i-1})|.$$
Consequently, $v_i\notin \mathrm{span}(B_{i-1})$ and $\dim\mathrm{span}(B_i)>\dim\mathrm{span}(B_{i-1})$. Due to the fact that $\dim\mathrm{span}(S)\le D$, this can occur at most $D$ times. Therefore, \[\P(E) \le \P\left(\sum_{i=1}^m\mathbf{1}_{E_i}\le D\right) =\P_{X\sim\mbox{Bin}(m,(1-\varepsilon)p)}(X\le D).\] as claimed.
\end{proof}

For example, if $S\subset \R^{dn}$ consists of the rows of $R:=R(K_n,\p)$, then the random variable $|S\cap\mathrm{span}_\R(S_p)|$, whose distribution is studied in the Lemma \ref{lem:process}, is equal to the size of $C_d(G(n,p))$. It turns out that the bound on $|C_d(G(n,p))|$ given by this lemma is not sufficient for our purposes, and we need to simultaneously apply it on $C_{d,A}(G(n,p))$ for all sets $A$ that are not too large.

\begin{lemma}\label{lem:large_wide_closure}
    Fix $1/2>\varepsilon>0$. Let $p=p(n)\ge \log n / n$, $d=d(n)$ satisfying $d\le (1/2-\varepsilon)np$ and $G\sim G(n,p)$. Then, a.a.s., $|C_{d,A}(G)|\ge \varepsilon n^2/15$ for every $A\subset [n]$ of size $|A|\le 0.9 n$.
\end{lemma}

\begin{proof}
Let $A\subset [n]$ of size $a$. By \eqref{eq:dim_A}, we can apply Lemma \ref{lem:process} to \\ $S=\left\{P_{W_A}\bv_f~:~f\notin\binom A2\right\}$, $m=\binom n2-\binom{a}2$ and $D=d(n-a)$. In such a case, $|S\cap\mathrm{span}_{\R}(S_p)|=|C_{d,A}(G)|$ and we obtain
\[
\P(|C_{d,A}(G)|<\varepsilon m) \le \P(X\le D),
\]
where $X\sim\mathrm{Bin}(m,(1-\varepsilon)p).$ 

A straightforward computation yields that $D\le (1-\varepsilon /2)\E X$ for all sufficiently large $n$. Hence, by Chernoff's inequality \eqref{eq:cher_non_sharp},
$$\P(X\le D) \le \exp(-(\varepsilon/2)^2 m(1-\varepsilon)p/2).$$
In addition, $m>n^2/15$ since $a\le 0.9n$. Therefore,
\begin{align*}
\P(|C_{d,A}(G)|< \varepsilon n^2/15)&~\le \P(|C_{d,A}(G)|< \varepsilon m)
\\ &~\le \exp(-\varepsilon^2(1-\varepsilon) n\log n/120).
\end{align*}
The proof is concluded by taking the union bound over the (less than $2^n$) sets $A$ of size $a\le 0.9n$.
\end{proof}

The next step in our argument is to use the fact that the graphs $C_{d,A}(G)$ for every $|A|
\le 0.9n$ are dense, in order to establish the existence of a large clique in $C_d(G)$.

We will need the following folklore fact from graph theory.
\begin{fact}\label{fct:folk}
    An $a$-vertex $b$-edge graph has an induced subgraph whose minimum degree is at least $b/a.$
\end{fact}

\begin{lemma}\label{lem:clique_in_closure}
Fix $1/2>\varepsilon>0$. Let $\log n /n \le p = o(n^{-1/2})$, $d=d(n)$ satisfying $d\le (1/2-\varepsilon)np$ and $G\sim G(n,p)$. Then, a.a.s., $C_d(G)$ contains a clique with at least $0.9n$ vertices.
\end{lemma}
\begin{proof}
    Condition on the a.a.s. events from Lemma \ref{lem:large_wide_closure} and Claim \ref{clm:Gnp_matching}. Let $A$ denote the largest clique in $C_d(G)$ and suppose, in contradiction, that $|A|<0.9n$. By Lemma \ref{lem:large_wide_closure}, there are at least $\varepsilon n^2/15$ edges in $C_{d,A}(G)$. In addition, since $A$ induces a clique in $C_d(G)$, these edges also belong to $C_d(G)$ by Claim \ref{clm:CdA}. We claim that less than $d|A^c|$ of these edges connect a vertex from $A$ to a vertex from $A^c$. Indeed, otherwise there exists a vertex $v\in A^c$ with $d$ neighbors from $A$ in the graph $C_d(G)$ whence, by Fact \ref{fct:heavy_vertex}, $A\cup\{v\}$ also induces a clique in $C_d(G)$ --- contradicting the maximality of $A$.
    
    Therefore, $A^c$ induces at least $\varepsilon n^2/15-o(n^{3/2})$ edges in $C_d(G)$. By Fact \ref{fct:folk}, there is a subset of $A^c$ that induces a subgraph $H$ of $C_d(G)$ with minimum degree
    \[
    \delta(H) \ge \frac{\varepsilon n^2/15-o(n^{3/2})}{|A^c|} \ge (1+o(1))\varepsilon n/15.
    \]
     The graph $H$ is also a closed graph since it is an induced subgraph of $C_d(G)$. In addition, its minimum degree is larger than $d(d+1)=o(n)$. Therefore, by Lemma \ref{lem:vil}, there exists a vertex $v \in H$ whose neighbor set $B$ induces a clique in $H$. 

     Consequently, $B$ is a clique in $C_d(G)$ that is disjoint from $A$. Moreover, by maximality, $|A|\ge |B|\ge \delta(H)\ge (1+o(1))\varepsilon n/15$. Therefore, since $d=o(n^{1/2})$, by Claim \ref{clm:Gnp_matching} there is a matching of size $\binom {d+1}2$ between $A,B$ in $G$. Hence, by Claim \ref{clm:matching}, the union $A\cup B$ also induces a clique in $C_d(G)$. This contradicts the maximality of $A$, and we deduce that $|A|\ge 0.9n$ as claimed.
\end{proof}

\begin{remark}
If the probabilistic argument in Lemma \ref{lem:process} was applied {\em only} to show that $C_d(G)$ a.a.s. has size $~\Omega(\varepsilon n^2)$, our reasoning would merely guarantee that it a.a.s. contains a clique of size $O(\sqrt{\varepsilon}n)$. In contrast, Lemma \ref{lem:clique_in_closure} establishes the existence of a clique of size $0.9n$, which is independent of $\varepsilon$. This underscores the importance of applying Lemma \ref{lem:process} to $C_{d,A}(G)$ for \emph{every}  $A$ satisfying $|A| \leq 0.9n$.

\end{remark}

\subsection{Proofs of main theorems}

In the previous subsection we have established the existence of a large clique in the $d$-rigidity closure of $G(n,p)$. To derive $d$-rigidity we need to show that $C_d(G)$ is an $n$-vertex clique. We argue that for any large clique $A$ in $C_d(G)$ which does not contain all the vertices,  there is a vertex $v\in B:=A^c$ that is adjacent to at least $d$ vertices in $A$ --- whence $A\cup\{v\}$ is an even larger clique in the closure. 

To prove this we show that either there is a vertex in $B$ with (substantially) more than $d$ neighbors in  $A$, or --- applicable only in the proof of Theorem \ref{thm:delta} --- there is an isolated vertex $v$ in the subgraph of $G$ induced by $B$, whence $v$ has at least $\delta(G)$ neighbors in $A$.

\begin{proof}[Proof of Theorem \ref{thm:delta}]
First, if $o(1)>p>(1+\varepsilon)C_*\log n / n$ then by Claim \ref{clm: min_deg_small_p} and the discussion below it there exists $\delta=\delta(\varepsilon)>0$, where $\delta\to 1$ as $\varepsilon\to\infty$, such that a.a.s., $\delta(G) > (1+\delta)np/2$. Hence, since a.a.s. $|E(G)|=(1+o(1))pn^2/2$, we have that
\begin{equation}\label{eq:Edel}
|E(G)| < \delta(G)n - \binom{\delta(G) +1}2,    
\end{equation}
and $G$ does not have enough edges to be $\delta(G)$-rigid. We note that if $\Omega(1)\le p\le 1- \omega(n^{-2})$ \eqref{eq:Edel} follows as well by standard arguments.

For the main part of the proof, we start by noting that since $\delta(G)$-rigidity is not a monotone property we cannot assume that $p$ is of the form $c\log n /n $ for some constant $c>0$, which slightly complicates the technicalities of our argument.

Let $p=p_n<(1-\varepsilon)C_*\log n/n$. It is well-known that if $(n-1)p\le \log n$ then a.a.s. $\delta(G)$ is either $0$ or $1$ and, in addition, $G$ is connected (i.e., $1$-rigid) if and only if $\delta(G)=1$ ~\cite{bollobas1998random}. 

Therefore we may assume that $c=c_n:=(n-1)p/\log n \in (1,(1-\varepsilon)C_*)$. Recall that $a=a_n$ is the smallest positive root of $\varphi_c(t)$ defined in \eqref{eq:phi}. We set the stage by selecting some parameters. Let $1>\beta>\varphi_{0.9}(1/2)\approx 0.894$, and we choose a constant $\delta>0$ independent of $n$ such that for every $n$,
\begin{enumerate}[(i)]
    \item $a+\delta < c/2$, and
    \item $\varphi_{c}(a+\delta)<(1-\beta)/4$.
\end{enumerate}
Such a constant $\delta>0$ exists since $a$ is bounded away from $c/2$ in the interval $c\in (1,(1-\varepsilon)C_*)$.

\begin{claim}\label{clm:inner}
The following event occurs a.a.s.: for every vertex subset $B$ in $G$ of size $1<b\le 0.1n$, either
\begin{enumerate}[(i)]
    \item there is a vertex $v\in B$ with at least $(a+\delta)\log n$ neighbors in $G$ outside $B$, or
    \item $B$ induces less than $b/2$ edges in $G$.
\end{enumerate}    
\end{claim}
\begin{proof}
    We bound the probability that a set $B$ of size $b$ violates both conditions. Note that the first condition deals with edges between $B$ and its complement, while the second deals with edges within $B$, and therefore they are independent. In addition, for each of the $b$ vertices in $B$, the events that each of them violate the first condition are independent. Therefore the probability of the event $E_B$ that $B$ violates both conditions is 
    \begin{equation}\label{eq:prob_expression}
        \P_{Y\sim \mathrm{Bin}(n-b,p)}(Y<(a+\delta)\log n)^b\cdot
    \P_{Z\sim \mathrm{Bin}\left(\binom b2,p\right)}(Z\ge b/2).
    \end{equation}
We give different bounds for this probability in two regimes of $b$. 

First, if $b\le n^{\beta}$ then $n-b=n(1+o(1))$ and  by Chernoff's inequality \eqref{eq:cher_main},
\[
\P(Y<(a+\delta)\log n) \le n^{\varphi_c(a+\delta)-1+o(1)}
\le n^{(1-\beta)/4-1+o(1)}
\]
where the last inequality is by the second assumption on $\delta.$ Additionally, we find by Chernoff's inequality \eqref{eq:cher_non_sharp_upper} that
\[
\P(Z\ge b/2)\le(ebp)^{b/2}<n^{-(1-\beta+o(1))b/2}.
\]
Here the last inequality follows from $b\le n^\beta$ and $p=O(\log n/n).$
Using these two inequalities and $\binom nb<n^b$ we find that
\begin{equation}\label{eq:small_sum}
    \sum_{B:1\le |B|\le n^{\beta}}\P(E_B)\le  \sum_{b=1}^{n^{\beta}} n^{-(1-\beta+o(1))b/4} \to 0,\quad\mbox{as $n\to \infty$.}
\end{equation}

On the other hand, if $n^{\beta}\le b \le 0.1 n$, we have that $n-b\ge 0.9n$. Hence,
\begin{align*}
    \P(Y<(a+\delta)\log n) &~\le 
\P_{Y'\sim \mathrm{Bin}(0.9n,p)}(Y'<(a+\delta)\log n)\\
&~\le n^{\varphi_{0.9c}(a+\delta)-1+o(1)}\\
&~\le n^{\varphi_{0.9}(1/2)-1+o(1)}.
\end{align*}
Here the second inequality follows by Chernoff's inequality and the last transition is derived from our first assumption on $\delta$, monotonicity  of $\varphi$:
\[
\varphi_{0.9c}(a+\delta)\le \varphi_{0.9c}(c/2) \le \varphi_{0.9}(1/2).
\]
Therefore, using the last inequality and $\binom nb\le (en/b)^b\le n^{(1-\beta+o(1))b}$ we find
\begin{equation}\label{eq:large_sum}
    \sum_{B:n^\beta\le |B|\le 0.1n}\P(E_B)\le  \sum_{b=n^\beta}^{0.1n} n^{(\varphi_{0.9}(1/2)-\beta+o(1))b} \to 0,
\end{equation}
as $n\to\infty$ by our assumption on $\beta$.

The claim is derived by combining \eqref{eq:small_sum} and \eqref{eq:large_sum} that show that the expected number of sets violating both conditions vanishes.
\end{proof}

To conclude the proof, we condition on the following three events that occur a.a.s. and show that they deterministically imply $\delta(G)$-rigidity.
\begin{enumerate}
    \item $\delta(G)<(a+\delta/2)\log n$, which occurs a.a.s. by Claim \ref{clm: min_deg_small_p},
    \item $C_d(G)$ contains a clique with $0.9n$ vertices where $d=\lceil(1/2-\delta/(2c))np\rceil$, which occurs a.a.s. by Lemma \ref{lem:clique_in_closure}, and
    \item the event in Claim \ref{clm:inner}.
\end{enumerate}
Our first assumption on $\delta$ is that $a+\delta<c/2$. Therefore, 
$$\delta(G)<(a+\delta /2)\log n < (c/2-\delta /2)\log n \le d$$
Hence, by the monotonicity of the rigidity closure with respect to dimension, we derive that $C_{\delta (G)}(G)$ also contains a clique with at least $0.9n$ vertices. 

Let $A$ be the largest clique in this closure and suppose, in contradiction that $B=A^c$ is non-empty. If $B$ satisfies the first condition in Claim \ref{clm:inner}, then there is a vertex $v\notin A$ with more than $\delta(G)$ neighbors from $A$. Hence, by Fact \ref{fct:heavy_vertex}, $A\cup \{v\}$ is also a clique in $C_{\delta (G)}(G)$ which contradicts the maximality of $A$. Otherwise, $B$ satisfies the second condition in Claim \ref{clm:inner}. In such a case, there is a vertex $v$ in $B$ that is isolated in the subgraph of $G$ induced by $B$. Consequently, all its neighbors in $G$ belong to $A$. But since $\delta(G)$ is the minimum degree in $G$, this means that $v$ has at least $\delta(G)$ neighbors in $A$, and we apply Fact \ref{fct:heavy_vertex} again to contradict the maximality of $A$.

In conclusion, $C_{\delta(G)}(G)$ is a clique and therefore $G$ is $\delta(G)$-rigid, as claimed.
\end{proof}

The proof of Theorem \ref{thm:np/2} is very similar and slightly simpler.
\begin{proof}[Proof of Theorem \ref{thm:np/2}]
First, if $d>(1/2+\varepsilon)np$ then a.a.s. $G$ does not have enough edges to be $d$-rigid. Next, suppose $d<(1/2-\varepsilon)np$. By Lemma \ref{lem:clique_in_closure} we have that a.a.s. $C_d(G)$ contains a clique of $0.9n$ vertices. To prove the theorem we claim that a.a.s. for every vertex subset $B$ in $G$ of size $|B|=b\le 0.1n$ there is a vertex $v\in B$ with $d$ neighbors outside of $B$. The theorem follows from the claim using Fact \ref{fct:heavy_vertex} as in the proof of Theorem \ref{thm:delta}, so we proceed directly to the proof of the claim.

Let $Y\sim\mathrm{Bin}(n-b,p)$. By combining Chernoff's inequality with the bounds $\E[Y]\ge 0.9C_*\log n$ and $d/\E[Y] <(1/2-\varepsilon)/0.9$, we find that there exists some constant $\beta>0$ such that
$
\P(Y<d) \le n^{-\beta}. 
$
Therefore, 
\begin{equation}\label{eq:big_sum_2}
    \sum_{b=n^{1-\beta/2}}^{0.1n} \binom nb (\P(Y<d))^b \le 
    \sum_{b=n^{1-\beta/2}}^{0.1n} n^{(-\beta/2+o(1))b}\to 0,
\end{equation}
as $n\to\infty$. Here we use $\binom nb\le (en/b)^b\le n^{(\beta/2+o(1))b}$.

If $1\le b<n^{1-\beta/2}$, then $\alpha:=d/\E[Y]<1/2-\varepsilon+o(1)$ whence by Chernoff's inequality,
\[
\P(Y<d) \le \left(\frac{e^{-(1-\alpha)}}{\alpha^\alpha}\right)^{np} \le n^{\varphi_{C_*}(\alpha C_*)-1},
\]
where the last inequality uses $np\ge C_*\log n$. Note that $\alpha < 1/2$ implies that $\varphi_{C_*}(\alpha C_*)<0$ since $\varphi_{C_*}(C_*/2)=0$. Therefore, using $\binom nb<n^b$,
\begin{equation}\label{eq:small_sum_2}
    \sum_{b=1}^{n^{1-\beta/2}} \binom nb (\P(Y<d))^b \le 
    \sum_{b=1}^{n^{1-\beta/2}} n^{(\varphi_{C_*}(\alpha C_*))b} \to 0
\end{equation}
as $n\to\infty$.

The claim is derived by combining \eqref{eq:big_sum_2} and \eqref{eq:small_sum_2} that show that the expected number of sets violating the condition in the claim vanishes.
\end{proof}

\section{Discussion and open problems}\label{sec:open}
Several open problems and related developments suggest themselves:
\begin{enumerate}
    \item The sharp form of Conjecture \ref{conj:KLM} remains open for $\Omega(n^{-1/2})\le p < 1$. Our proof for Theorem \ref{thm:np/2} uses the assumption $p=o(n^{-1/2})$ in two critical ways: (i) using Vill\'any's Lemma \ref{lem:vil}, which shows that a dense closed graph contains a large clique, and (ii) establishing the existence of a matching of $\binom{d+1}2$ edges between any two disjoint vertex sets of linear size in $G$. Since the first version of this paper, Krivelevich, Lew and Michaeli~\cite{krivelevich2026combinatorial} gave a partial answer, up to the multiplicative constant, in the sparse regime $p=o(1)$ by proving that for $p\ge 2\log n/n$, the random graph $G(n,p)$ is a.a.s.\ $\lfloor cnp\rfloor$-rigid for some absolute constant $c>0$. 
    \item The dense regime, where $p$ is bounded away from $0$ and $1$, presents significant challenges. In the first version of this paper, we asked whether $G(n,1/2)$ is a.a.s.\ $(\varepsilon n)$-rigid for some $\varepsilon>0$. This question was answered affirmatively by Jord\'an, Liu and Vill\'anyi~\cite{jordan2025degree}, who proved that $G(n,1/2)$ is a.a.s.\ $d$-rigid for $d\sim 7n/32$. Determining the correct constant remains an interesting open problem.
    \item In the regime $p<p_c$, we are able to find the precise maximum dimension $d=\delta(G)$ such that $G$ is $d$-rigid. We conjecture, and this is verified by some numerical experiments, that for $p>p_c$, a.a.s., $G$ is $d$-rigid if and only if the event $|E(G)|\ge dn-\binom{d+1}{2}$ occurs. However, our probabilistic approach (Lemma \ref{lem:process}) is meaningful only if the dimension $d$ is bounded away from $np/2$, and is therefore too weak for this conjecture.
    \item Finally, we note that our proof, in its current form, applies to all $1$-extendable abstract rigidity matroids. Do the theorems hold true for \emph{all} abstract rigidity matroids? Note that the results from \cite{LNPR} hold for all abstract rigidity matroids. 
    \end{enumerate}

\medskip

\noindent \textbf{Acknowledgments}
We thank Orit Raz for her help with the proof of Claim \ref{clm:matching}.

\bibliographystyle{abbrv}
\bibliography{rig-2}

\end{document}